\documentclass[a4j,12pt]{article}
\usepackage{amsmath,amsthm,amssymb,amscd,ascmac, amsfonts}
\usepackage{mathrsfs}
\usepackage{stmaryrd}
\usepackage{braket}
\usepackage{accents}
\usepackage{url}
\allowdisplaybreaks[4]

\usepackage[dvips]{graphicx,color,psfrag}

\makeatletter

\newtheorem*{multitheorem}{\variable@name}

\theoremstyle{definition}
\newcommand{\variable@name}{Theorem}
\newtheorem*{multiproclaim}{\variable@name}

\theoremstyle{plain}
\newtheorem{thm}{Theorem}
\newtheorem{prop}[thm]{Proposition}
\newtheorem{lem}[thm]{Lemma}
\newtheorem{cor}[thm]{Corollary}

\theoremstyle{definition}

\newtheorem{rem}[thm]{Remark}

\begin{document}
\title{Rational values of powers of trigonometric functions}
\author{Genki Shibukawa}
\date{
\small MSC classes\,:\,11J72, 11R18, 33B10}
\pagestyle{plain}

\maketitle

\begin{abstract}
We extend the theorem by Olmsted (1945) and Carlitz-Thomas (1963) on rational values of trigonometric functions to powers of trigonometric functions. 
\end{abstract}

\section{Introduction}
Throughout the paper, we denote the ring of rational numbers by $\mathbb{Q}$, the ring of real numbers by $\mathbb{R}$, the set of positive rational numbers by $\mathbb{Q}_{>0}$ and a $m$th root of unity by $\zeta_{m}:=e^{\frac{2\pi \sqrt{-1}}{m}}$. 
Olmsted \cite{O} and Carlitz-Thomas \cite{CT} determined all rational values of trigonometric functions. 
\begin{thm}[Olmsted (1945), Carlitz-Thomas (1963)]
\label{thm:main thm}
If $\theta \in \mathbb{Q}$, then the only possible rational values of the trigonometric functions are: 
$$
\sin {(\pi \theta )},\,\cos {(\pi \theta )} =0,\,\pm \frac{1}{2}, \,\pm1 \,\, ; \,\, \tan {(\pi \theta )}=0, \,\pm1.
$$
\end{thm}
By this Theorem\,\ref{thm:main thm} and well-known facts
$$
\cos{(\pi \theta )}^{2}
   =
   \frac{1+\cos{(2\pi \theta )}}{2}, \quad 
   \tan{(\pi \theta )}^{2}
   =
   \frac{1}{\cos{(\pi \theta )}^{2}}-1,
$$
we have the following result immediately. 
\begin{cor}
\label{thm:square cos}
If $\theta \in \mathbb{Q}$ and $\cos{(\pi \theta )}^{2} \in \mathbb{Q}$, then the only possible values of the trigonometric functions are: 
$$
\sin {(\pi \theta )},\,\cos {(\pi \theta )} =0, \, \pm \frac{1}{2}, \, \pm \frac{1}{\sqrt{2}}, \, \pm \frac{\sqrt{3}}{2}, \, \pm1 
\,\, ; \,\, \tan {(\pi \theta )}=0, \,\, \pm \frac{1}{\sqrt{3}}, \,\, \pm 1, \,\, \pm \sqrt{3}.
$$
\end{cor}
In this note, we prove Theorem\,\ref{thm:gen main thm} and Theorem\,\ref{thm:gen main thm 2}, which are generalizations of Theorem\,\ref{thm:main thm} and Corollary\,\ref{thm:square cos}. 
\begin{thm}
\label{thm:gen main thm}
If $N\geq 3$ and $\alpha $ is a positive rational number such that $\alpha^{\frac{1}{N}},\ldots,\alpha^{\frac{N-1}{N}}\not\in \mathbb{Q}$, then for any positive integer $m$, we have
$$
\sqrt[N]{\alpha }\not\in \mathbb{Q}(\zeta_{m}). 
$$
In particular, there is no $\theta \in \mathbb{Q}$ such that $\cos{(\pi \theta )},\cos{(\pi \theta )}^{2},\ldots ,\cos{(\pi \theta )}^{N-1} \not\in \mathbb{Q}$ and $\cos{(\pi \theta )}^{N} \in \mathbb{Q}$ (resp. $\tan{(\pi \theta )},\tan{(\pi \theta )}^{2},\ldots ,\tan{(\pi \theta )}^{N-1} \not\in \mathbb{Q}$ and $\tan{(\pi \theta )}^{N} \in \mathbb{Q}$).
\end{thm}
\begin{thm}
\label{thm:gen main thm 2}
If there exists a positive integer $n$ and $\theta \in \mathbb{Q}$ such that $\cos{(\pi \theta )}^{n} \in \mathbb{Q}$ (resp. $\tan{(\pi \theta )}^{n} \in \mathbb{Q}$), then the only possible values of the trigonometric functions are: 
\begin{equation}
\label{eq:cos sin results}
\sin {(\pi \theta )},\,\cos {(\pi \theta )}
=
\begin{cases}
0, \,\, \pm \frac{1}{2}, \,\, \pm 1 & (n:\text{odd}) \\
0, \,\, \pm \frac{1}{\sqrt{2}}, \,\, \pm \frac{1}{2}, \,\, \pm \frac{\sqrt{3}}{2}, \,\, \pm 1 & (n:\text{even})
\end{cases}.
\end{equation}
resp. 
\begin{equation}
\label{eq:tan result}
\tan {(\pi \theta )}
=
\begin{cases}
0, \,\, \pm 1 & (n:\text{odd}) \\
0, \,\, \pm \frac{1}{\sqrt{3}}, \,\, \pm 1, \,\, \pm \sqrt{3} & (n:\text{even})
\end{cases}.
\end{equation}
\end{thm}
We note that one can easily verify the above theorems using the fundamental properties of Galois theory as in the following sections. 
But we could find no proof of these facts in print, and hence it will be of some interest to write down the proofs of these facts.

\section{Preliminaries}

To prove Theorem\,\ref{thm:gen main thm} and Theorem\,\ref{thm:gen main thm 2}, we list some fundamental facts of the cyclotomic fields and Kummer extension in this section. 
First we mention the Galois group $\mathrm{Gal}(\mathbb{Q}(\zeta_{n})/\mathbb{Q})$ (see \cite{L}). 
\begin{lem}
The degree of the cyclotomic extension $\mathbb{Q}(\zeta_{n})$ over $\mathbb{Q}$ is $[\mathbb{Q}(\zeta_{n}):\mathbb{Q}]=\varphi (n):=|\{1\leq a\leq n \mid \mathrm{gcd}(a,n)=1\}|$ and its Galois group $\mathrm{Gal}(\mathbb{Q}(\zeta_{n})/\mathbb{Q})$ is 
\begin{equation}
\begin{array}{ccccccc}
(\mathbb{Z}/n\mathbb{Z})^{\times } & \simeq & \mathrm{Gal}(\mathbb{Q}(\zeta_{n})/\mathbb{Q}) & \curvearrowright & \mathbb{Q}(\zeta_{n}) & \rightarrow & \mathbb{Q}(\zeta_{n}) \\
\rotatebox{90}{$\in$} &  & \rotatebox{90}{$\in$} & & \rotatebox{90}{$\in$} & & \rotatebox{90}{$\in$} \\
c & \mapsto & \tau _{c} & \curvearrowright & \zeta_{n} & \mapsto & \tau _{c}(\zeta_{n}):=\zeta_{n} ^{c}
\end{array}. \nonumber 
\end{equation}
In particular $\mathrm{Gal}(\mathbb{Q}(\zeta_{n})/\mathbb{Q})$ is an abelian extension, and its subfields $L \supset \mathbb{Q}$ are Galois and abelian extensions over $\mathbb{Q}$.
\end{lem}
Under the following, let $\alpha$ be a positive rational number such that 
\begin{equation}
\label{eq:sufficient cond of alpha}
\alpha^{\frac{1}{n}},\ldots,\alpha^{\frac{n-1}{n}}\not\in \mathbb{Q}
\end{equation}
and $K:=\mathbb{Q}(\sqrt[n]{\alpha }, \zeta_{n})$. 
\begin{prop}
\label{thm:basic prop}
{\rm{(1)}} The binomial polynomial $x^{n}-\alpha $ is irreducible over $\mathbb{Q}$ and $[\mathbb{Q}(\sqrt[n]{\alpha }):\mathbb{Q}]=n$.\\
{\rm{(2)}} For any $n\geq 2$, we have $\sqrt[n]{\alpha } \not\in \mathbb{Q}(\zeta_{n})$.
\end{prop}
\begin{proof}
{\rm{(1)}} 
We consider the following decomposition of the binomial polynomial: 
\begin{align}
\label{eq:binomial decompose}
x^{n}-\alpha 
   &=
   \prod_{i=1}^{n}(x-\sqrt[n]{\alpha }\zeta_{n}^{i})
   =
   f_{I}(x)f_{J}(x)  
\end{align}
where $I, J$ are subsets of $[n]:=\{1,2,\ldots, n\}$ such that
\begin{align}
\label{eq:disjoint I J}
[n]=I \sqcup J, \quad I\not=\emptyset , \quad J\not=\emptyset , \quad I\cap J=\emptyset
\end{align}
and  
\begin{align}
f_{I}(x)
   &:=
   \prod_{i \in I}(x-\sqrt[n]{\alpha }\zeta_{n}^{i})
   =
   x^{|I|}+\cdots +(-1)^{|I|}\alpha ^{\frac{|I|}{n}}\prod_{i \in I}\zeta_{n}^{i} \in \mathbb{Q}[x], \nonumber \\
f_{J}(x)
   &:=
   \prod_{j \in J}(x-\sqrt[n]{\alpha }\zeta_{n}^{j})
   =
   x^{|J|}+\cdots +(-1)^{m}\alpha ^{\frac{|J|}{n}}\prod_{j \in J}\zeta_{n}^{j} \in \mathbb{Q}[x]. \nonumber 
\end{align}
If $x^{n}-\alpha $ is reducible over $\mathbb{Q}$, then there exist $I, J \subset [n]$ such that satisfy the condition (\ref{eq:disjoint I J}) and $f_{I}(x), f_{J}(x)$ are rational coefficient polynomials. 
Since the constant term of $f_{J}(x)$ is real, the product 
$
\prod_{j \in J}\zeta_{n}^{j} 
$
is real and 
$
\left|\prod_{j \in J}\zeta_{n}^{j} \right|=1.
$
Thus we have 
$$
\prod_{j \in J}\zeta_{n}^{j}=\pm 1.
$$
From the assumption $\alpha ^{\frac{|J|}{n}} \not\in \mathbb{Q}$, the constant term of $f_{J}(x)$
$$
(-1)^{m}\alpha ^{\frac{m}{n}}\prod_{j \in J}\zeta_{n}^{j}
   =
   \pm (-1)^{m}\alpha ^{\frac{m}{n}}
$$
is irrational. 
It is a contradiction.\\
{\rm{(2)}} Assume $\sqrt[n]{\alpha } \in \mathbb{Q}(\zeta_{n})$. 
Then we have the contradiction 
$$
n=[\mathbb{Q}(\sqrt[n]{\alpha }):\mathbb{Q}]\leq [\mathbb{Q}(\zeta_{n}):\mathbb{Q}]=\varphi (n)=n\prod_{p\mid n}\left(1-\frac{1}{p}\right)<n. 
$$
\end{proof}

\begin{lem}
\label{thm:Kummer Galois}
{\rm{(1)}} If $n=p$ is a odd prime, then the binomial type polynomial $x^{p}-\alpha $ is irreducible over $\mathbb{Q}(\zeta_{p})$ and $[K:\mathbb{Q}(\zeta_{p})]=p$. Its Galois group $\mathrm{Gal}(K/\mathbb{Q})$ is
\begin{equation}
\begin{array}{ccccccc}
\mathbb{Z}/p\mathbb{Z} \rtimes (\mathbb{Z}/p\mathbb{Z})^{\times } & \simeq & \mathrm{Gal}(K/\mathbb{Q}) & \curvearrowright & K & \rightarrow & K \\
\rotatebox{90}{$\in$} &  & \rotatebox{90}{$\in$} & & \rotatebox{90}{$\in$} & & \rotatebox{90}{$\in$} \\
(1,1) & \mapsto & \sigma & \curvearrowright & \sqrt[p]{\alpha } & \mapsto & \sigma (\sqrt[p]{\alpha }):=\zeta _{p}\sqrt[p]{\alpha } \\
 & & & & \zeta_{n} & \mapsto & \sigma (\zeta_{p} ):=\zeta _{p} \\
(0,c) & \mapsto & \tau_{c} & \curvearrowright & \sqrt[p]{\alpha } & \mapsto & \tau_{c} (\sqrt[p]{\alpha } ):=\sqrt[p]{\alpha } \\
 & & & & \zeta_{p} & \mapsto & \tau_{c} (\zeta_{p} ):=\zeta _{p}^{c}
\end{array}. \nonumber 
\end{equation}
In particular, $\tau_{c} \sigma =\sigma ^{c}\tau_{c}$ and for $n\geq 3$ the Galois group $\mathrm{Gal}(K/\mathbb{Q})$ is non-abelian.\\
{\rm{(2)}} For any $n\geq 3$, the Galois group $\mathrm{Gal}(K/\mathbb{Q})$ is non-abelian.
\end{lem}
\begin{proof}
{\rm{(1)}} We consider the factorization (\ref{eq:binomial decompose}) $x^{p}-\alpha =f_{I}(x)f_{J}(x)$ again. 
By 
$
\mathrm{gcd}(|I|,p)=1
$ 
and 
$
\mathrm{gcd}(|J|,p)=1,
$ 
$\mathbb{Q}(\alpha ^{\frac{|I|}{p}})$ and $\mathbb{Q}(\alpha ^{\frac{|J|}{p}})$ contain $\sqrt[p]{\alpha }$. 
Hence, from Proposition \ref{thm:basic prop} {\rm{(2)}}, $\alpha ^{\frac{|I|}{p}}$ and $\alpha ^{\frac{|J|}{p}}$ are not contained in $\mathbb{Q}(\zeta_{p})$. 
Therefore the binomial polynomial $x^{p}-\alpha $ is irreducible over $\mathbb{Q}(\zeta_{p})$ and $[K:\mathbb{Q}(\zeta_{p})]=p$.\\
{\rm{(2)}} When a odd prime $p$ divides $n$, $K$ contains a non-abelian Galois extension $\mathbb{Q}(\sqrt[p]{\alpha },\zeta_{p})$ over $\mathbb{Q}$, so the Galois group $\mathrm{Gal}(K/\mathbb{Q})$ is non-abelian. 
If $n=2^{m}$ $(m\geq 2)$, then $K$ contains a non-abelian Galois extension $\mathbb{Q}(\sqrt[4]{\alpha },\zeta_{4})$ over $\mathbb{Q}$ and $\mathrm{Gal}(K/\mathbb{Q})$ is also non-abelian.
\end{proof}
\begin{rem}
Lemma \ref{thm:Kummer Galois} {\rm{(1)}} is not true in general. 
For example, when $n=8$ and $\alpha =2$ the polynomial $x^{8}-2$ is reducible over $\mathbb{Q}(\zeta_{8})$ even though $2^{\frac{1}{8}},\ldots,2^{\frac{7}{8}}$ are irrational. 
In fact
$$
x^{8}-2
   =
   (x^{4}-\sqrt{2})(x^{4}+\sqrt{2})
   =
   (x^{4}-\zeta_{8}-\zeta_{8}^{-1})(x^{4}+\zeta_{8}+\zeta_{8}^{-1}).
$$
\end{rem}

\section{Proof of Theorem\,\ref{thm:gen main thm}}
Assume there exists $N\geq 3$, $\alpha \in \mathbb{Q}_{>0}$ and a positive integer $m$ such that $\alpha^{\frac{1}{N}},\ldots,\alpha^{\frac{N-1}{N}}\not\in \mathbb{Q}$ and 
$$
\sqrt[N]{\alpha }
   \in 
   \mathbb{Q}(\zeta_{m}).
$$
Then 
$$
\mathbb{Q}(\sqrt[N]{\alpha })
   \subset 
   \mathbb{Q}(\zeta_{m}).
$$
Although $\mathbb{Q}(\sqrt[N]{\alpha })$ is not a Galois extension over $\mathbb{Q}$, $K$ is a Galois extension over $\mathbb{Q}$ and  
$$
K
   \subset 
   \mathbb{Q}(\zeta_{m},\zeta_{N})
   \subset 
   \mathbb{Q}(\zeta_{mN}).
$$
Further the field $K$ is a subfield of $\mathbb{Q}(\zeta_{mN})$ and the Galois group $\mathrm{Gal}(K/\mathbb{Q})$ is a normal subgroup of $\mathrm{Gal}(\mathbb{Q}(\zeta_{mN})/\mathbb{Q})$:
$$
\mathrm{Gal}(K/\mathbb{Q}) \triangleleft \mathrm{Gal}(\mathbb{Q}(\zeta_{mN})/\mathbb{Q})\simeq (\mathbb{Z}/mN\mathbb{Z})^{\times }.
$$
However the Galois group $\mathrm{Gal}(K/\mathbb{Q})$ is non-abelian. 
It is a contradiction. 
Then for any positive integer $m$, 
\begin{equation}
\label{eq:irrationality of cos}
\sqrt[N]{\alpha }\not\in \mathbb{Q}(\zeta_{m}).
\end{equation}

For the above $N\geq 3$ and positive rational number $\alpha \in \mathbb{Q}_{>0}$, assume there exists $\theta \in \mathbb{Q}$ such that 
$$
\cos{(\pi \theta )}=\sqrt[N]{\alpha }.
$$
By $\theta \in \mathbb{Q}$, there exists a positive integer $m$ such that 
$$
\sqrt[N]{\alpha }
   =
   \cos{(\pi \theta )}
   \in 
   \mathbb{Q}(\zeta_{m}).
$$
For $N\geq 3$, it is contrary to (\ref{eq:irrationality of cos}). 
For $\tan{(\pi \theta )}$, one can prove similarly.
\begin{rem}
From Theorem\,\ref{thm:gen main thm}, there is no cyclotomic field over $\mathbb{Q}$ containing $n$th root of a positive rational number $\alpha $ with $\sqrt[n]{\alpha } \not\in \mathbb{Q}$, for any $n\geq 3$. 
On the other hand, from Gauss sum's formulas \cite{BEW}
\begin{align}
\sum_{k=0}^{m-1}\zeta _{m}^{k^{2}}
   &=
   \frac{1+\sqrt{-1}}{2}(1+(-\sqrt{-1})^{m})\sqrt{m}
   =\begin{cases}
   (1+\sqrt{-1})\sqrt{m} & (m\equiv 0 \,\mathrm{mod}\,4) \\
   \sqrt{m} & (m\equiv 1 \,\mathrm{mod}\,4) \\
   0 & (m\equiv 2 \,\mathrm{mod}\,4) \\
   \sqrt{-1}\sqrt{m} & (m\equiv 3 \,\mathrm{mod}\,4)
   \end{cases}, \nonumber \\
\zeta_{4}
   &=
   \sqrt{-1}, \quad 
\zeta_{8}+\zeta_{8}^{-1}
   =
   \sqrt{2}, \nonumber 
\end{align} 
there exists a positive integer $m$ such that $\sqrt{\alpha } \in \mathbb{Q}(\zeta_{m})$, for any $\alpha \in \mathbb{Q}$.  
\end{rem}

\section{Proof of Theorem\,\ref{thm:gen main thm 2}}
Since the proof of (\ref{eq:tan result}) is similar to (\ref{eq:cos sin results}), we only prove (\ref{eq:cos sin results}). 
The cases of $n=1$ and $n=2$ are Theorem\,\ref{thm:main thm} and Corollary\,\ref{thm:square cos} respectively. 
For $n\geq 3$, the following three cases are possible: \\
\indent
 1) $\cos{(\pi \theta )}\in \mathbb{Q}$ and $\cos{(\pi \theta )}^{2}\in \mathbb{Q}$, \\
\indent
 2) $\cos{(\pi \theta )}\not\in \mathbb{Q}$ and $\cos{(\pi \theta )}^{2}\in \mathbb{Q}$, \\
\indent
 3) $\cos{(\pi \theta )}\not\in \mathbb{Q}$ and $\cos{(\pi \theta )}^{2}\not\in \mathbb{Q}$.\\
In the case of 1), from Theorem\,\ref{thm:main thm} and Corollary\,\ref{thm:square cos}, the possible values of $\cos{(\pi \theta )}$ (or $\sin{(\pi \theta )}$) are 
$
0,\,\pm \frac{1}{2}, \,\pm1. 
$
Similarly, for the case of 2), the possible values of $\cos{(\pi \theta )}$ are 
$
\pm \frac{1}{\sqrt{2}}, \,\, \pm \frac{\sqrt{3}}{2}. 
$
Finally, the case of 3) is impossible from Theorem\,\ref{thm:gen main thm} {\rm{(1)}}. 
Then we obtain the conclusion (\ref{eq:cos sin results}).

\section*{Acknowledgement}
We would like to thank Professor Takashi Taniguchi (Kobe University) for his comments on cyclotomic and Kummer extensions. 


\bibliographystyle{amsplain}

\noindent 
Department of Mathematics, Graduate School of Science, Kobe University, \\
1-1, Rokkodai, Nada-ku, Kobe, 657-8501, JAPAN\\
E-mail: g-shibukawa@math.kobe-u.ac.jp

\end{document}